\documentclass[12pt]{amsart}
\usepackage{amscd,amsmath,amsthm,amssymb}
\usepackage{pstcol,pst-plot,pst-3d}%\usepackage[T1]{fontenc}
\usepackage{color}
\usepackage{pstricks}
\newpsstyle{fatline}{linewidth=1.5pt}
\newpsstyle{fyp}{fillstyle=solid,fillcolor=verylight}
\definecolor{verylight}{gray}{0.97}
\definecolor{light}{gray}{0.9}
\definecolor{medium}{gray}{0.85}
\definecolor{dark}{gray}{0.6}

 \def\NZQ{\mathbb}               % the font for N,Z,Q,R,C

 \def\ZZ{{\NZQ Z}}

 \def\FF{{\NZQ F}}

 \def\frk{\mathfrak}               % font for "Fraktur"

 \def\mm{{\frk m}}
 
 \def\nn{{\frk n}}
 \def\G{{\mathcal G}}

  \def\Mc{{\mathcal M}}
  \def\S{{\mathcal S}}

 \def\ab{{\mathbf a}}

 \def\opn#1#2{\def#1{\operatorname{#2}}} % to make operators
 \opn\chara{char} \opn\length{\ell} \opn\pd{pd} \opn\rk{rk}
 \opn\projdim{proj\,dim} \opn\injdim{inj\,dim} \opn\rank{rank}
 \opn\depth{depth} \opn\grade{grade} \opn\height{height}
 \opn\embdim{emb\,dim} \opn\codim{codim}
 
 \opn\Tr{Tr} \opn\bigrank{big\,rank}
 \opn\superheight{superheight}\opn\lcm{lcm}
 \opn\trdeg{tr\,deg}%\emph{
 \opn\reg{reg} \opn\lreg{lreg} \opn\ini{in} \opn\lpd{lpd}
 \opn\size{size} \opn\sdepth{sdepth}
 \opn\link{link}\opn\fdepth{fdepth}\opn\lex{lex}
 \opn\tr{tr}
 \opn\type{type}
 \opn\Borel{Borel}
 \opn\div{div} \opn\Div{Div} \opn\cl{cl} \opn\Cl{Cl}
 \opn\Spec{Spec} \opn\Supp{Supp} \opn\supp{supp} \opn\Sing{Sing}
 \opn\Ass{Ass} \opn\Min{Min}\opn\Mon{Mon}
 \opn\Ann{Ann} \opn\Rad{Rad} \opn\Soc{Soc}
 \opn\Im{Im} \opn\Ker{Ker} \opn\Coker{Coker} \opn\Am{Am}
 \opn\Hom{Hom} \opn\Tor{Tor} \opn\Ext{Ext} \opn\End{End}
 \opn\Aut{Aut} \opn\id{id}
 
 \opn\nat{nat}
 \opn\pff{pf}%   \pf exists already
 \opn\Pf{Pf} \opn\GL{GL} \opn\SL{SL} \opn\mod{mod} \opn\ord{ord}
 \opn\Gin{Gin} \opn\Hilb{Hilb}\opn\sort{sort}
 \opn\PF{PF}\opn\Ap{Ap}
 \opn\aff{aff} \opn
 \con{conv} \opn\relint{relint} \opn\st{st}
 \opn\lk{lk} \opn\cn{cn} \opn\core{core} \opn\vol{vol}  \opn\inp{inp} \opn\nilpot{nilpot}
 \opn\link{link} \opn\star{star}\opn\lex{lex}\opn\set{set}
 \opn\width{wd}
 \opn\Fr{F}
 \opn\QF{QF}
 \opn\G{G}
 \opn\type{type}\opn\res{res}
 \opn\gr{gr}
 
 \def\pot#1#2{#1[\kern-0.28ex[#2]\kern-0.28ex]}

 \opn\dirlim{\underrightarrow{\lim}}
 \opn\inivlim{\underleftarrow{\lim}}
 \let\union=\cup
 \let\sect=\cap

 \let\iso=\cong

 \let\Dirsum=\bigoplus
 
 \def\Implies{\ifmmode\Longrightarrow \else
         \unskip${}\Longrightarrow{}$\ignorespaces\fi}
 \def\implies{\ifmmode\Rightarrow \else
         \unskip${}\Rightarrow{}$\ignorespaces\fi}
 \def\iff{\ifmmode\Longleftrightarrow \else
         \unskip${}\Longleftrightarrow{}$\ignorespaces\fi}

 \let\:=\colon
 \newtheorem{Theorem}{Theorem}[section]
 \newtheorem{Lemma}[Theorem]{Lemma}
 \newtheorem{Corollary}[Theorem]{Corollary}
 \newtheorem{Proposition}[Theorem]{Proposition}
 
 \newtheorem{Remarks}[Theorem]{Remarks}
 
 \newtheorem{Examples}[Theorem]{Examples}

 \let\epsilon\varepsilon
 \let\kappa=\varkappa
 \def\qed{\ifhmode\textqed\fi
       \ifmmode\ifinner\quad\qedsymbol\else\dispqed\fi\fi}
 \def\textqed{\unskip\nobreak\penalty50
        \hskip2em\hbox{}\nobreak\hfil\qedsymbol
        \parfillskip=0pt \finalhyphendemerits=0}
 \def\dispqed{\rlap{\qquad\qedsymbol}}
 
 \opn\dis{dis}
 \def\pnt{{\raise0.5mm\hbox{\large\bf.}}}
 
 \opn\Lex{Lex}

\begin{document}
%\linenumbers
\title {Freiman ideals}

\author {J\"urgen Herzog and  Guangjun Zhu$^{^*}$}

\address{J\"urgen Herzog, Fachbereich Mathematik, Universit\"at Duisburg-Essen, Campus Essen, 45117
Essen, Germany} \email{juergen.herzog@uni-essen.de}

\address{Guangjun Zhu, School of Mathematical Sciences, Soochow University,
 Suzhou 215006, P. R. China}\email{zhuguangjun@suda.edu.cn}

%\dedicatory{To my beloved friend }用来说明这篇文章献给谁的

\begin{abstract}
In this paper we study the Freiman inequality for the minimal number of generators of the square of an equigenerated  monomial ideal. Such an ideal is called a Freiman ideal if equality holds in the Freiman inequality. We classify all Freiman ideals of maximal  height, the Freiman ideals of certain classes of principal Borel ideals, the Hibi ideals which are Freiman, and classes of Veronese type ideals which are Freiman.
\end{abstract}

\thanks{The paper was written while the second author was visiting the Department of Mathematics of University
Duisburg-Essen. She want to express her thanks for the hospitality.\\
\hspace{5mm}* Corresponding author.}

%		13H10   	Special types (Cohen-Macaulay, Gorenstein, Buchsbaum, etc.)
%		13D02   	Syzygies, resolutions, complexes
%		05E40   	Combinatorial aspects of commutative algebra
%		16S36   	Ordinary and skew polynomial rings and semigroup rings

%		14M25   	Toric varieties, Newton polyhedra [See also 52B20]
%		13A02   	Graded rings
%		13F20   	Polynomial rings and ideals; rings of integer-valued polynomials
%		13A18   	Valuations and their generalizations
%		06A11   	Algebraic aspects of posets


\maketitle

\setcounter{tocdepth}{1}
%\tableofcontents

\section*{Introduction}

Let $I$ be an equigenerated monomial ideal with analytic spread $\ell(I)$. It has been shown in \cite[Theorem 1.9]{HMZ} that  $\mu(I^2)\geq l(I)\mu(I)-{l(I)\choose 2}.$ Here $\mu(J)$ denotes the minimal number of generators of a graded  ideal $J$. This inequality is the  consequence of a well-known theorem from additive number theory, due to Freiman \cite{F1}. It should be noted that the above lower bound for the  minimal number of the generators of  the square  of a monomial ideal is no longer valid, if $I$ is not equigenerated. Indeed,  for each $m\geq 6$ there exists a monomial ideal $I$ in two variables with $\mu(I)=m$  and $\mu(I^2)=9$, see \cite{EHM}.

We call an equigenerated  monomial ideal $I$ a {\em Freiman ideal} (or simply Freiman), if $\mu(I^2)= l(I)\mu(I)-{l(I)\choose 2}.$ It is the aim of this paper to classify all Freiman ideals within given families of equigenerated monomial ideals.

\medskip
In Section~1 we analyze the Freiman inequality from the view  point of  commutative algebra. Let $I$ be a graded  ideal in the polynomial ring $S=K[x_1,\ldots,x_n]$ over a field $K$. The fiber cone $F(I)$  of $I$ is the standard graded $K$-algebra $\Dirsum_{k\geq 0}I^k/\mm I^k$, where $\mm=(x_1,\ldots,x_n)$ is the graded maximal ideal of $S$.  The Hilbert series of $F(I)$ is of the form $Q(t)/(1-t)^{\ell(I)}$,  where $Q(t)=\sum_{i\geq 0}h_it^i$ is a polynomial with integer coefficients $h_i$. It has been noticed in \cite[Corollary 2.6]{HMZ} the surprising fact,  that, as a consequence of  the Freiman inequality,   one has $h_2\geq 0$ for any equigenerated monomial ideal. Moreover, $h_2=0$ if and only if $I$ is a Freiman ideal.

We do not know whether the fiber cone of any Freiman ideal is Cohen--Macaulay. For all families of Freiman ideals considered in this paper,  the fiber cone of a Freiman ideal is  Cohen--Macaulay.

Assuming the fiber cone of a given equigenerated monomial ideal $I$ is Cohen--Macaulay, it is not hard to show (see Proposition~\ref{hyper} and  Corollary~\ref{powerFreiman}) that the following conditions are equivalent:
\begin{enumerate}
\item[(a)] $I$ is Freiman.
\item[(b)] $F(I)$ has minimal multiplicity.
\item[(c)] There exists an integer $k\geq 2$ such that  $$\mu(I^{k})={\ell(I)+k-2\choose k-1}\mu(I)-(k-1){\ell(I)+k-2\choose k}.$$
\item[(d)]  For all integer $k\geq 1$ one has  $\mu(I^{k})={\ell(I)+k-2\choose k-1}\mu(I)-(k-1){\ell(I)+k-2\choose k}$.
\end{enumerate}

\medskip
Moreover,  if $I$ is Freiman and $F(I)$ is Cohen--Macaulay, then  the reduction number of $I$ is one, and  $I$ is level. If, in addition, we assume that $I$ is  equimultiple, that is, if $\ell(I)=\height(I)$, then   these  properties actually characterize Freiman ideals, see Corollary~\ref{level} and Corollary~\ref{fishsticks}.

\medskip
In Section~\ref{sec2} we give a complete characterization of Freiman ideals of height $n$  in the polynomial ring  $S=K[x_1,\ldots, x_n]$ with $n\geq 2$, where $K$ is a field. In Theorem~\ref{Thm3} it is first observed that if $I\subset S$ is an equigenerated monomial ideal of height $n$ which is generated in degree $d$,  then $I$ is Freiman if and only if $I^2=(x_1^d,\ldots,x_n^d)I$.  This result is then used to prove Theorem~\ref{pseudo}, which is the main theorem of this section. In order to formulate this main result, we need the  following definition. Let $I\subset S$ be a monomial ideal. The unique minimal set of monomial generators of $I$ is denoted by $G(I)$. Let $G(I)=\{u_1,\ldots,u_m\}$ and let $q\geq 1$ be  an integer. We denote by $I^{[q]}$ the monomial ideal with $G(I^{[q]})=(u_1^q,\ldots,u_m^q)$. The ideal $I^{[q]}$ is called the {\em $q$th pseudo-Frobenius power} of $I$. It is obvious that $I$ is Freiman if and only if $I^{[q]}$ is Freiman.  Thus,  in the classification of Freiman ideals it suffices to consider monomial ideals which are not proper pseudo-Frobenius powers of other monomial ideals. With this concept and notation introduced,  we are ready to present Theorem~\ref{pseudo}, which says the following:  Let $n, d\geq 2$   be two integers and  $I\subset K[x_1,\ldots,x_n]$  an equigenerated monomial ideal of height $n$  generated in  degree $d$. Suppose that   $I$ is not a proper pseudo-Frobenius power of another monomial ideal.   Then $I$ is Freiman, if and only if,  up to a relabeling of the variables,  $I=(x_1,\ldots,x_r)^d+(x_{r+1}^d,\ldots,x_n^d)$ with $r\leq \min\{3,n\}$ if $d=2$,  and $r\leq \min\{2,n\}$ if $d\geq 3$.  In the case that $d=1$, it follows that  $I=(x_1,\dots, x_n)$. This ideal is obviously Freiman.

In Corollary~\ref{product} we discuss the question in which cases the product of two equigenerated monomial ideals $I$ and $J$ of height $n$ in $K[x_1,\ldots,x_n]$ is a Freiman ideal. It turns out that this is never the case, if $n\geq 4$, while if $n=3$, the product $IJ$ is Freiman if and only if $I=J=(x_1^d,x_2^d,x_3^d)$ for some integer $d\geq 1$.

\medskip
In Section~\ref{sec3} we study classes of monomial ideals which naturally appear in combinatorics and geometry, and ask which of them are Freiman. The classes of ideals considered here are the principal Borel ideals, the Hibi ideals and the ideals of Veronese type. For a principal Borel ideal $I$  with Borel generator $u=x_ix_j$  with $i\leq j$, it is shown in Theorem~\ref{Borel}(a) that  $I$ is Freiman, if and only if $j\leq 3$, or $j>3$ and $i\leq 2$.
For principal Borel ideals of degree $d\geq 3$ we only have very partial results. The reason is that checking the Freiman condition for such ideals leads to difficult numerical problems, which at this moment, we are not able to handle. May be there exists another approach to these problems which we are aware of at present. Actually we expect that if $u$ is a monomial of degree $d\geq 3$ such that $x_1$ does not divide $u$, and $I$ is the ideal whose Borel generator is $u$, then $I$ is Freiman if and only if $u=x_2^{d-1}x_j$ and $2\leq j$. The ``if" part of this expected result is shown in Theorem~\ref{Borel}(b), and both directions for $n=3$ are shown in Theorem~\ref{Thm4}.

For Hibi ideals we have a complete answer. Recall that if $P$ is a finite poset and $\mathcal{I}(P)$ is the set of poset ideals of $P$, then the {\em Hibi ideal} $H_P$
is the monomial ideal in the polynomial ring  $S=K[\{x_p,y_p\}_{p\in P}]$,  whose generators are the monomials  $u_{I}=(\prod\limits_{p\in \mathcal{I}(P)}x_p)(\prod\limits_{p\in P\setminus \mathcal{I}(P)}y_p)$ with $I\in \mathcal{I}(P)$. It is shown in Theorem~\ref{hibifreiman} that
$H_P$ is  Freiman  if and only if there exists $p\in P$ such that $P\setminus\{p\}$ is a chain. A typical example of such a poset, but not the only one, is a poset consisting of the chain plus an extra element which is incomparable with all the elements of this chain.

For Veronese type ideals we give a complete answer in three cases. Given positive integers  $n$, $d$,  and a sequence $\ab$ of integers $1\leq a_1\leq a_2\leq \cdots \leq a_n\leq d$ with $\sum_{i=1}^na_i>d$, one defines the monomial ideal  $I_{\ab,d}\subset S=K[x_1,\ldots,x_n]$ with
\[
G(I_{\ab,d})=\{x_1^{b_1}x_2^{b_2}\cdots x_n^{b_n}\; \mid \; \sum_{i=1}^nb_i=d \text{ and  $b_i\leq a_i$ for $i=1,\ldots,n$}\}.
\]
Ideals of this form are called {\em ideals of Veronese type}.
If $a_i=d$ for $i=1,\ldots,n$ then $I_{\ab,d}=(x_1,\ldots,x_n)^d$, and $(x_1,\ldots,x_n)^d$ is Freiman if and only if  $d=1$,  $n\leq 2$, or $n= 3$ and $d=2$, see  Theorem~\ref{lexsegment}.

If $a_i=1$ for $i=1,\ldots,n$, then $I_{\ab,d}$ is the so-called squarefree Veronese ideal which, as is common,  we also denote by $I_{n,d}$. In this case $I_{n,d}$ is  Freiman  if and only if  $d=1$ or  $d=n-1$, see Theorem~\ref{Veronese2}. For the proof we use the result of De Negri and Hibi which says that the fiber cone of a Veronese type ideal is Cohen--Macaulay, and that $I_{n,d}$ is level if and only if $d=1$, $d=n-1$, or $d\geq 2$ and $n=2d-1$, $n=2d$ or $n=2d+1$. Thus we may apply our result from Section~\ref{sec1} which says that Freiman ideals whose fiber cone is Cohen--Macaualy must be level, and it remains to check only these cases.

Katzman \cite{Ka} determined the multiplicity of the fiber cone of a Veronese type ideal. The formula is rather complicated, but in the case that $a_i=d-1$ for $i=1,\ldots,n$, the formula simplifies and one has $e(F(I_{\ab,d}))=d^{n-1}-n$. Now by using the fact, shown in Section~\ref{sec1}, that Freiman ideals, whose fiber cone is Cohen--Macaulay, must have minimal multiplicity, we show in Theorem~\ref{Veronese3} that $I_{\ab,d}$ is Freiman if and only if $n=2$, or $n=3$ and  $d=2$.

The results of this paper show that Freiman ideals are rather rare. Thus one can hope that a classification of all principal Borel Freiman ideals and all Veronese type Freiman ideals is possible. We are not so optimistic that Freiman ideals in general can be classified. The Hibi ideals which are Freiman show that  it would be hard to find a general pattern for all Freiman ideals.

\medskip

\section{The Freiman inequality}
\label{sec1}

Let $K$ be a field and $S=K[x_1,\ldots,x_n]$ the polynomial ring in $n$ indeterminates over $K$, and  let $I\subset S$ be a graded ideal.  The minimal number of generators of $I$ will be denoted by $\mu(I)$. The  ideal $I$ is called {\em equigenerated},  if all elements  of a minimal set of generators of  $I$ have the same degree.

An ideal  $J\subseteq I$  is  called  a {\it reduction} of  $I$ if $I^{k+1}=JI^{k}$
 for some nonnegative integer $k$.  The
{\it reduction number} of $I$ with respect to $J$ is defined to be   $$r_{J}(I)=\min\{k \mid I^{k+1}=JI^{k}\}.$$
A reduction $J$ of $I$ is called a {\em minimal reduction}  if  it does not properly contain any other  reduction of $I$.  The {\it reduction number} of $I$ is defined to be the number  $$r(I)=\min\{ r_{J}(I)\mid J\ \mbox{ is a minimal reduction of}\ I\}.$$

By Northcott and Rees \cite{NR} it is known that  $\mu(J)=\ell(I)$ for  every
minimal reduction $J$ of $I$, if $K$ is infinite. Here $\ell(I)$ denotes  the {\it analytic spread} of $I$, which  is defined to be the Krull dimension of the fiber cone $F(I)=\Dirsum_{k\geq 0} I^k/\mm I^k$.

It is known  that $\height(I)\leq \ell(I) \leq \min\{\dim(S),\mu(I)\}$. The ideal  $I$ is called  {\it equimultiple},  if $\height(I)=l(I)$.

\medskip
A well-known theorem of Freiman \cite{F1} implies the following result \cite{HMZ}:

\begin{Theorem}
\label{freiman}
Let $I\subset S$ be an equigenerated monomial ideal. Then $$\mu(I^2)\geq l(I)\mu(I)-{l(I)\choose 2}.$$
\end{Theorem}

It is of interest to  know when we have equality in the above theorem. To have a short name, we say that an equigenerated monomial ideal $I$ is a {\em Freiman ideal} (or simply {\em Freiman}), if $\mu(I^2)=\ell(I)\mu(I)-{\ell(I)\choose 2}$. For simplicity we set $$\Delta(I)=\mu(I^2)-\ell(I)\mu(I)+{\ell(I)\choose 2}$$
for any graded ideal $I\subset S$.

\begin{Remarks}
\label{useful}
{\em
(a) Let $\Hilb_{F(I)}(t)= Q(t)/(1-t)^{\ell(I)}$ with $Q(t)=1+h_1t +h_2t^2+\cdots$ be the Hilbert series of the fiber cone $F(I)$ of $I$, where $I\subset S$ is a graded ideal of $S$.
In \cite{HMZ} it has been noticed that $\Delta(I)=h_2$. Thus an equigenerated monomial ideal is Freiman  if and only if $h_2=0$.

(b) Let $J\subset K[x_1,\ldots,x_r]$ with $r<n$  be an equigenerated monomial ideal, and let $I=JS$. Then $I$ is Freiman if and only if $J$ is Freiman. Indeed, $\mu(I)=\mu(J)$ and $\mu(I^2)=\mu(J^2)$.
}
\end{Remarks}

\medskip
Let $R$ be a standard graded $K$-algebra. We denote by $e(R)$ the {\em multiplicity} of $R$ and by $\embdim(R)$ the {\em embedding dimension} of $R$.
By Abhyankar \cite{A} it is known that
\[
\embdim(R)\leq e(R)+\dim R-1,
\]
if $R$ is a domain. The same inequality holds if $R$ is Cohen--Macaulay \cite{Sa}. The $K$-algebra $R$ is said to have {\em minimal multiplicity} if
$\embdim(R)= e(R)+\dim R-1.$

\begin{Proposition}
\label{hyper}
Let $I$ be an equigenerated monomial ideal. Suppose that $F(I)$ is Cohen--Macaulay.  Then $I$ is  Freiman, if and only if  $F(I)$ has minimal multiplicity. This is for example the case, if $F(I)$ is a hypersurface ring defined by a quadratic polynomial.
\end{Proposition}

\begin{proof}
The Hilbert series of the fiber cone $F(I)$ of $I$ is of the form
\[
\Hilb_{F(I)}(t)=\sum\limits_{k\geq 0}\mu(I^{k})t^{k}=\frac{1+h_{1}t+h_{2}t^2+h_{3}t^3+h_{4}t^4+\cdots}{(1-t)^{\ell(I)}}.
\]
Since $F(I)$ is Cohen--Macaulay, it follows that $\sum_{i\ge 2}h_i\ge 0$ and that $$h_1=\embdim F(I)-\dim F(I).$$  It follows that $F(I)$ has minimal multiplicity if and only if $h_i=0$ for all $i\ge 2$. For a Cohen--Macaulay ring we have $h_2=0$ if and only if $h_i=0$ for all $i\ge 2$. This yields the desired result.
\end{proof}

\begin{Examples}
\label{okgoogle}
{\em Let $I=(x^2,y^2,z^3,xy)$, $J=(x^2,y^2,z^2,xy,xz)$ and $L=(x,y,z)^2$. Then $I$ is  Freiman, and  $F(I)$   is a hypersurface ring defined by a quadric.   $J$ is not Freiman, because $\ell(J)=3$, and $13=\mu(J^2)>3\mu(J)-3=12$. Finally,  $L$ is  Freiman, and  $F(L)$ is Cohen--Macaulay with minimal multiplicity.

\medskip
In Theorem~\ref{lexsegment} it is shown  for $n\ge 3$, the ideal $(x_1,\ldots, x_n)^m$ is Freiman if and only if  $m=1$, $n\leq 2$, or $n=3$ and $m=2$. }
\end{Examples}

Let  $T=K[z_1,\ldots,z_m]$ be a polynomial ring over a field $K$,  $J\subset T$ a graded ideal with graded minimal free $T$-resolution $\FF$. The ideal $J$ is said to have a $d$-linear resolution if $F_i=T(-d-i)^{\beta_i}$ for all $i$, and $T/J$ is said to be {\em level}, if $T/J$ is Cohen--Macaulay,   and the last module in the graded free resolution of $J$ is generated in a single degree.

\begin{Corollary}
\label{level}
Let $I\subset S$ be a Freiman ideal, and suppose that $F(I)$ is Cohen--Macaulay. Then the defining ideal of $F(I)$ has a $2$-linear resolution and $F(I)$ is level.
\end{Corollary}

\begin{proof}
We may assume that $K$ is an infinite field. Then there exists a maximal regular sequence $y_1,\ldots, y_{\ell(I)}$ of linear forms in $F(I)$. Let $\overline{F(I)}=F(I)/(y_1,\ldots, y_{\ell(I)})F(I)$. Since $F(I)$ and $\overline{F(I)}$  have the same multiplicity, and since by Proposition~\ref{hyper},  $F(I)$ has minimal multiplicity, it follows that $\bar{\nn}^2=0$, where $\bar{\nn}$ denotes the graded maximal ideal of $\overline{F(I)}$. This implies that $\overline{F(I)}$  is level and has a $2$-linear resolution. Since $\overline{F(I)}$ is  obtained from $F(I)$ by reduction modulo a regular sequence of linear forms, the desired result follows.
\end{proof}

\medskip
In the following theorem we consider more generally when for a graded ideal $I$ we have $\Delta(I)=0$.

\begin{Theorem}
\label{equivalent}
Let $I\subset S$ be a graded ideal.
\begin{enumerate}
\item[ (a)] If  $\Delta(I)=0$ and $F(I)$ is Cohen--Macaulay, then $r(I)=1$.
\item[(b)] If $I$ is   equimultiple
and $r(I)=1$, then $\Delta(I)=0$ and $F(I)$ is Cohen--Macaulay.
\end{enumerate}
\end{Theorem}

\begin{proof}
(a) Without loss of generality we may assume that  $K$ is infinite. We set $\ell=\ell(I)$. The case $\ell=0$ is  trivial.   So we may assume that $\ell\geq 1$.   Since $F(I)$ is Cohen--Macaulay, we can choose  an $F(I)$-regular sequence $y_{1},\ldots,y_{\ell}\in F(I)_{1}$, where
 $y_{i}=f_{i}+\mm I$ with $f_i\in I$.

Let $\Hilb_{F(I)}(t)=Q_{F(I)}(t)/(1-t)^\ell$ with  $\deg Q(t)=q$ and $Q_{F(I)}(t)=\sum_{i=0}^qh_it^i$.  Since we assume that $F(I)$ is Cohen--Macaulay, it follows that $h_i>0$  for $i=0,\ldots,q$. By Remark~\ref{useful}(a), $\Delta(I)=h_{2}$.
Thus our assumption implies that $h_2=0$, and hence $h_i=0$ for all $i\geq 2$.

 Let $J=(f_{1},\ldots, f_{\ell})$ and  $\overline {F(I)}=F(I)/(J+\frak{m}I)F(I)$. Since $y_{1},\ldots,y_{\ell}$ is a regular sequence on $F(I)$, it follows that $Q_{F(I)}(t)=Q_{\overline {F(I)}}(t)$. This implies that $\overline {F(I)}_{i}=0$ for $i\geq 2$. Equivalently, $$I^{2}=JI+\frak{m}I^{2}.$$
 Thus,  Nakayama's lemma yields $I^{2}=JI$, as desired.

 (b) By Shah \cite[Corollary 1(a)]{S1},  our assumptions imply that $F(I)$ is Cohen--Macaulay. Let $J=(f_1,\ldots,f_{\ell})$ with $f_i$ as in part (a). Then the elements $y_i=f_i+\mm I$ in $F(I)_1$ form a regular sequence and $\overline{F(I)}_2=0$, where $\overline{F(I)}=F(I)/(y_1,\ldots,y_\ell)F(I)$. Since $F(I)$ and $\overline {F(I)}$ have  the same $h$-vector, it follows that $h_2=\bar{h}_2=0$. Thus $\Delta(I)=0$.
\end{proof}

\begin{Corollary}
\label{fishsticks}
Let $I$ be an equigenerated monomial ideal.
\begin{enumerate}
\item[(a)] If $I$ is  Freiman  and $F(I)$ is Cohen--Macaulay, then $r(I)=1$.
\item[(b)] If $I$ is equimultiple and $r(I)=1$, then $I$ is  Freiman  and $F(I)$ is Cohen--Macaulay.
\end{enumerate}
\end{Corollary}

In the next result we consider the minimial number of generators of the powers of a graded ideal.

\begin{Proposition} \label{zhu} Let  $I\subset S$ be a graded  ideal with   analytic spread $l=\ell(I)$. Then
\begin{enumerate}
\item[(a)]  $\mu(I^{k})\geq {l+k-2\choose k-1}\mu(I)-(k-1){l+k-2\choose k}$ for all $k\geq 1$, if  $h_{i}\geq 0$ for all $i\geq 2$.

\item[(b)] $\mu(I^{k})={l+k-2\choose k-1}\mu(I)-(k-1){l+k-2\choose k}$ for all $k\geq 1$, if and only if  $h_{i}=0$ for all $i\geq 2$.

\item[(c)] Assume  in addition that $F(I)$ is Cohen--Macaulay. Then the following conditions are equivalent:
\begin{enumerate}
\item[(i)] $\Delta(I)=0$.
\item[(ii)]  $\mu(I^{k})= {l+k-2\choose k-1}\mu(I)-(k-1){l+k-2\choose k}$ for all $k\geq 1$.
\item[(iii)]  $\mu(I^{k})= {l+k-2\choose k-1}\mu(I)-(k-1){l+k-2\choose k}$ for some $k\geq 2$.
\end{enumerate}
\end{enumerate}
\end{Proposition}

\begin{proof} (a) For the fiber cone of the ideal $I$, we have
\begin{eqnarray*}
\Hilb_{F(I)}(t)&=&\sum\limits_{k\geq 0}\mu(I^{k})t^{k}=\frac{1+h_{1}t+h_{2}t^2+h_{3}t^3+h_{4}t^4+\cdots}{(1-t)^{\ell}}
\end{eqnarray*}
\begin{eqnarray*}
&=&(1+h_{1}t+h_{2}t^2+\cdots)(1+\ell t+{\ell+1\choose 2}t^2+{\ell+2\choose 3}t^3+\cdots)\\
&=&1+(h_{1}+\ell)t+(h_{1}\ell+{\ell+1\choose 2}+h_{2})t^2+\cdots\\
&+&[{\ell+k-1\choose k}+\sum\limits_{i=1}^{k}{\ell+k-i-1\choose k-i}h_{i}]t^k+\cdots.
\end{eqnarray*}
It follows  that
 $$\mu(I)=h_{1}+\ell,\ \mbox{ and}\  \ \mu(I^2)=h_{1}\ell+{\ell+1\choose 2}+h_{2}=\ell\mu(I)-{\ell\choose 2}+h_{2}.$$
Moreover,  for all $k\geq 3$ we have
\begin{eqnarray*}\mu(I^k)&=&{\ell+k-1\choose k}+{\ell+k-2\choose k-1}h_{1}+{\ell+k-3\choose k-2}h_{2}+\cdots+{\ell\choose 1}h_{k-1}+h_{k}\\
&=&{\ell+k-1\choose k}+{\ell+k-2\choose k-1}(\mu(I)-\ell)+\sum\limits_{i=2}^{k}{\ell+k-i-1\choose k-i}h_{i}   \\
&=&{\ell+k-2\choose k-1}\mu(I)-(\ell-1){\ell+k-2\choose k}+\sum\limits_{i=2}^{k}{\ell+k-i-1\choose k-i}h_{i}.
\end{eqnarray*}
Since   $h_{i}\geq 0$ for all $i\geq 2$, it follows that $\sum\limits_{i=2}^{k}{\ell+k-i-1\choose k-i}h_{i}\geq 0$ for all $i\geq 2$. This yields the desired  conclusion.

(b) The proof of (a) shows that $\mu(I^{k})={\ell+k-2\choose k-1}\mu(I)-(k-1){\ell+k-2\choose k}$ for all $k\geq 1$,
if  $h_{i}=0$ for all $i\geq 2$.

Conversely, suppose that $\mu(I^{k})={\ell+k-2\choose k-1}\mu(I)-(k-1){\ell+k-2\choose k}$ for all $k\geq 1$. Then
\begin{eqnarray*}
\Hilb_{F(I)}(t)&=&1+\sum\limits_{k\geq 1}[{\ell+k-2\choose k-1}\mu(I)-(k-1){\ell+k-2\choose k}]t^k\\
&=&\frac{1+[\mu(I)-\ell]t}{(1-t)^{\ell}}.
\end{eqnarray*}
This shows that $h_{i}=0$ for all $i\geq 2$.

(c) (i)\implies (ii): If $\Delta(I)=0$, then $h_2=0$.  Since $F(I)$ is Cohen--Macaulay, it follows that $h_i=0$ for $i\geq 2$. Therefore, the assertion follows from (b).

(ii)\implies (iii): is trivial.

(iii)\implies (i):  Let $k\geq 2$. By the proof of (a) we know that
\[
 \mu(I^k)- [{\ell+k-2\choose k-1}\mu(I)-(\ell-1){\ell+k-2\choose k}]=\sum\limits_{i=2}^{k}{\ell+k-i-1\choose k-i}h_{i}.
 \]
 Therefore, (iii) implies that $\sum\limits_{i=2}^{k}{\ell+k-i-1\choose k-i}h_{i}=0$. In particular, $h_2=0$  and hence $\Delta(I)=0$.
\end{proof}

\begin{Corollary}
\label{powerFreiman}
{\em Proposition~\ref{zhu} implies that an equigenerated monomial ideal $I$ is Freiman if and only if
$\mu(I^{k})= {l+k-2\choose k-1}\mu(I)-(k-1){l+k-2\choose k}$ for all or just some   $k\geq 2$.}
\end{Corollary}

\medskip

\medskip

\section{Freiman ideals of maximal height}
\label{sec2}

The following result gives a characterization of Freiman ideals of maximal height.

\begin{Theorem}
\label{Thm3}
Let $I\subset K[x_1,\ldots,x_n]$ be an equigenerated  monomial ideal generated in  degree $d$ with  $\height I=n$. Then  $I$ is  Freiman  if and only if $I^2=JI$,  where $J=(x_1^d,\ldots,x_n^d)$.
If the equivalent conditions hold, then $F(I)$ is Cohen--Macaulay.
\end{Theorem}

\begin{proof}

Assume that   $G(I)=\{u_1,\ldots, u_{n+t}\}$ where $u_i=x_i^{d}$  for  $1\leq i\leq n$, $u_{n+i}=x_{1}^{a_{i1}}x_{2}^{a_{i2}}\cdots x_{n}^{a_{in}}$  with $0\leq a_{ij}<d$  for   $1\leq j\leq n$ and $1\leq i\leq t$.
Thus $\mu(I)=n+t$ and $\ell(I)=n$.

Assume first that $I^2=JI$.  Then, since $I$ is  equimultiple,   Theorem \ref{equivalent}(b) implies that $I$ is  Freiman.

Conversely, suppose that  $I$ is Freiman. Then  $$\mu(I^2)=\ell(I)\mu(I)-{\ell(I)\choose 2}=n(n+t)-{n\choose 2}={n+1\choose 2}+nt.$$
Let $L=(u_{n+1},\dots, u_{n+t})$. Since $G(J^{2})\cap G(JL)=\emptyset$, we see that
$$\mu(JI)=\mu(J^2+JL)=\mu(J^2)+\mu(JL)={n+1\choose 2}+\mu(JL).$$
Note that $\mu(JL)=nt$, and hence $\mu(I^2)=\mu(JI)$. Indeed, suppose that $x_i^du_{n+j}=x_k^du_{n+l}$. If $i=k$, then $j=l$. On the other hand, if $i\neq k$, it follows that
$x_i^d$ divides $u_{n+l}$, a contradiction. Thus, $x_i^du_{n+j}=x_k^du_{n+l}$ if and only if $i=k$ and $j=l$.

 Since the monomials of $G(J^{2})$, $G(JL)$ and $G(I^{2})$
are all of degree $2d$  and $J^{2}+JL\subseteq I^{2}$, it follows that $I^{2}=JI$.

The remaining statement follows  from  Theorem \ref{equivalent}(b).
\end{proof}

Theorem~\ref{pseudo} will give  a full  characterization of  Freiman ideals of height $n$ in $S=K[x_1,\ldots,x_n]$. Before proving this theorem we need
two preliminary results.

\begin{Proposition}
\label{onedirection} Let $I\subset K[x_1,\ldots,x_m]$ be an equigenerated  monomial ideal generated in  degree $d$.  Let $K[y_1,\ldots,y_n]$ be another  polynomial ring and $J=(y_1^d,\ldots, y_n^d)$. Then  $I$ is Freiman if and only if $I+J$ is  Freiman.
\end{Proposition}
\begin{proof}It is enough to show that the assertion holds for $n=1$. Let $J=(y_1^d)$, then $\mu(J)=1$,  $\mu(I+J)=\mu(I)+1$, $\mu(JI)=\mu(I)$, $\ell(I+J)=\ell(I)+1$ and $(I+J)^2=I^2+JI+J^2$.
Since the monomials of $G(I^{2})$, $G(JI)$ and $G(J^{2})$ are pairwise distinct, it follows that
$$\mu((I+J)^2)=\mu(I^{2})+\mu(JI)+\mu(J^{2})
=\mu(I^{2})+\mu(I)+1.
$$
If $I$ is Freiman, then $$\mu(I^{2})=\ell(I)\mu(I)-{\ell(I)\choose 2}.$$
It follows that
\begin{eqnarray*}
\mu((I+J)^2)&=&\ell(I)\mu(I)-{\ell(I)\choose 2}+\mu(I)+1\\
&=&(\ell(I)+1)(\mu(I)+1)-{\ell(I)+1\choose 2}\\
&=&\ell(I+J)\mu(I+J)-{\ell(I+J)\choose 2}.
\end{eqnarray*}
This implies that  $I+J$ is  Freiman.

Conversely,  if   $I+J$ is  Freiman, then $$\mu((I+J)^2)=\ell(I+J)\mu(I+J)-{\ell(I+J)\choose 2},$$
i.e.,  $$\mu(I^{2})+\mu(I)+1=(\ell(I)+1)(\mu(I)+1)-{\ell(I)+1\choose 2}.$$
Thus $$\mu(I^{2})=\ell(I)\mu(I)-{\ell(I)\choose 2}.$$
This implies that  $I$ is  Freiman. \end{proof}

\begin{Theorem}
 \label{lexsegment}
 Let $m$ be a positive integer. Then  $I=(x_1,\ldots,x_n)^m$ is  Freiman if and only if  $m=1$, $n\leq 2$, or $n=3$ and  $m=2$.
\end{Theorem}
\begin{proof} We may assume that  $n\geq 3$ and $m\geq 2$, because otherwise the statement holds from Theorem \ref{Thm3}.  If $n\geq 4$, then  $x_1^{m-1}x_2^{m-1}x_3x_4\in I^2\setminus (x_1^m,\ldots,x_n^m)I$.   Hence  $I$ is not  Freiman by Theorem~\ref{Thm3}.
So we  only need  show that if $n=3$, then $I=(x_1,\ldots,x_n)^m$ is  Freiman if and only if $m=2$. In this case,
$$\mu(I)={3+m-1\choose m}={m+2\choose m} \ \text{and}\ \ \mu(I^2)={3+2m-1\choose 2m}={2m+2\choose 2m}.$$
It follows that
\[
\Delta (I)=\mu(I^2)-3\mu(I)+{3\choose 2}={2m+2\choose 2m}-3{m+2\choose m}+3=\frac{(m-1)(m-2)}{2}.
\]
Therefore,  $I=(x_1,x_2,x_3)^m$ is  Freiman if and only if $m=2$.
\end{proof}

\medskip
One more definition  and a simple observation is required before we can formulate Theorem~\ref{pseudo}, let $I$ be a monomial ideal and let $r\ge 1$ be an integer. Let $G(I)=\{u_1,\ldots,u_m\}$. We let $I^{[r]}$ be the monomial ideal with
$G(I^{[r]})=\{u_1^r,\ldots,u_m^r\}$. The monomial   ideal $I^{[r]}$ is called the {\em $r$th pseudo-Frobenius power} of $I$. It is called a {\em proper} pseudo-Frobenius power of $I$, if $r\geq 2$.

Since $\mu(I^k)=\mu((I^{[r]})^k)$ for all $k$, it follows that $I$ is Freiman if and only if $I^{[r]}$ is Freiman. It is also obvious that $F(I)\iso F(I^{[r]})$.

\begin{Theorem}
\label{pseudo}
Let $I\subset K[x_1,\ldots,x_n]$ with $n\geq 2$  be   a Freiman ideal  of height $n$  generated in  degree $d\geq 2$, and  suppose that $I$ is not a proper pseudo-Frobenius power of another monomial ideal. Then,  up to a relabeling of the variables,  $$I=(x_1,\ldots,x_r)^d+(x_{r+1}^d,\ldots,x_n^d)$$ with $r\leq \min\{3,n\}$ if $d=2$,  and $r\leq \min\{2,n\}$ if $d\geq 3$.
\end{Theorem}

\begin{proof}
It follows from Proposition~\ref{onedirection} and Theorem~\ref{lexsegment} that $I$ is Freimann, if $I=(x_1,\ldots,x_r)^d+(x_{r+1}^d,\ldots,x_n^d)$ with $r\leq \min\{3,n\}$ if $d=2$,  and $r\leq \min\{2,n\}$ if $d\geq 3$.

For the converse direction  we will use the result of Theorem~\ref{Thm3}  which says that $I$ is Freiman  if and only if $I^2=JI$, where $J=(x_1^d,\ldots,x_n^d)$. This implies in particular, that $I^k=J^{k-1}I$ for all $k\geq 2$.

Assume first that  $d=2$.  Suppose that $x_ix_j, x_kx_l\in I$ with pairwise different indices. Then $x_ix_jx_kx_l\in JI\subset J$, which is impossible. If $G(I)$ contains no mixed product, then $I=J$. If $G(I)$ contains one mixed product, say $x_1x_2$, then $I=(x_1,x_2)^2+(x_3^2,\ldots,x_n^2)$. If $G(I)$ contains two mixed products, they must have a common factor. Say, $x_1x_2,x_1x_3\in I$.  Then $x_1^2x_2x_3\in JI$. This implies that $x_2x_3\in I$,  and hence $x_1x_2, x_1x_3,x_2x_3\in I$. Thus in this case, $I=(x_1,x_2, x_3)^2+(x_4^2,\ldots,x_n^2)$.

Next we consider the case that $d\geq 3$. For a monomial $u\in S$ we set $\supp(u)=\{i\:\; x_i|u\},$ and  first show:
\begin{eqnarray}
\label{venlo}
\text{if $I$ is Freiman and there exists $u\in G(I)\setminus G(J)$, then $|\supp(u)|=2$.}
\end{eqnarray}
Indeed, let $u\in G(I)\setminus G(J)$ with  $|\supp(u)|=m\geq 3$. We may assume that $u=x_1^{a_1}\cdots x_m^{a_m}$ with $0<a_i<d$ for $i=1,\ldots,m$. Let $a_i(d-1)\equiv r_i\mod d$ with $0\leq r_i<d$.  It follows that $r_i=d-a_i$. Hence $(a_i-1)$ is the highest power of $x_i^d$ which divides $u^{d-1}$ and we get
\[
u^{d-1}=(x_1^d)^{a_1-1}\cdots (x_m^d)^{a_m-1}x_1^{d-a_1}\cdots x_m^{d-a_m}.
\]
It follows that $u^{d-1}\in J^{d-m}\setminus J^{d-m+1}$, because $(a_1-1)+\cdots+(a_m-1)=d-m.$  Since $u^{d-1}\in I^{d-1}=J^{d-2}I\subset J^{d-2}$, it follows that $m=2$.

Now  suppose there exists a monomial  $u\in G(I)$ with $|\supp u|=2$. We may assume that $\supp u=\{1,2\}$. Let $\S$ be  the set of all monomials $u$ in $G(I)$ with $\supp u=\{1,2\}$. Let $x_1^{d-i}x_2^i,x_1^{d-j}x_2^j\in \S$. If $i+j\leq d$, then, since $I^2=JI$, it follows that $x_1^{d-(i+j)}x_2^{i+j}\in I$, and if $i+j> d$, then $x_1^{d-(i+j-d)}x_2^{i+j-d}\in I$. This shows that the elements $i+dZ\in \ZZ/d\ZZ$ with $x_1^{d-i}x_2^i\in \S$ form a subgroup $U$ of $\ZZ/d\ZZ$. Any subgroup of $\ZZ/d\ZZ$ is cyclic. Let $0<a<d$ be  the smallest  integer with $a+d\ZZ\in U$.  Then $a$ divides $d$ and $x_1^{d-la}x_2^{la}\in I$ for $l=0,\ldots,r$, where $r=d/a$. Hence $(x_1^a,x_2^a)^r\subset I$ with $ar=d$ and $a$ minimal  with this property.

Suppose $G(I)$ contains another monomial with support $\{i,j\}\neq \{1,2\}$, say $v=x_i^{d-e}x_j^e$ with $0<e<d$. Suppose that $\{1,2\}\sect\{i,j\}=\emptyset$. Then $x_1^{d-a}x_2^ax_i^{d-e}x_j^e\in I^2\subset J$ with all exponents less than $d$, a contradiction. Hence $ \{i,j\}\sect\{1,2\}\neq \emptyset$, and we may assume that $\{i,j\}=\{2,3\}$. Then $(x_2^b,x_3^b)^s\subset I$ with $bs=d$ and $b$ minimal  with this property. Hence we have $(x_1^a,x_2^a)^r+(x_2^b,x_3^b)^s\subset I$. It follows that $w_1=x_1^ax_2^{2d-(a+b)}x_3^b=(x_1^ax_2^{d-a})(x_2^{d-b}x_3^b)\in I^2\subset J$. If $a+b>d$, then all exponents of $w_1$ are less than  $d$, a contradiction, and if $a+b<d$, then $x_1^ax_2^{d-(a+b)}x_3^b\in I$, contradicting the fact that $G(I)$ does not contain monomials whose support has more than $2$ elements. Thus we must have that $a+b=d$. This shows that $(x_2^b,x_3^b)^s=(x_2^a,x_3^a)^r$. It follows that $x_1^ax_2^{2d-2a}x_3^a\in I^2$. Since $d=ra$ with $r\geq 2$ it follows that $d-2a\geq 0$. Therefore, $w_2=x_1^ax_2^{d-2a}x_3^a\in I$. If $d>2a$, then $|\supp(w_2)|=3$, a contradiction. Hence, $d=2a$, and $(x_1^a,x_2^a)^2+(x_2^a,x_3^a)^2\subset I$. If there exists no other monomial in $G(I)$, whose cardinality is $2$. Then
\[
I=(x_1^a,x_2^a)^2+(x_2^a,x_3^2)^2+(x_3^{2a},\ldots, x_n^{2a}).
\]
Since we assume that $I$ is not a proper pseudo-Frobenius power, it follows that $d=2$, contradicting our assumption that $d\geq 3$.

By the above arguments, if there exists other elements $u\in G(I)$ with $\supp(u)=\{i,j\}$, then $\{i,j\}\sect\{1,2\}\neq  \emptyset$ and $\{i,j\}\sect\{2,3\}\neq  \emptyset$. Therefore, $\{i,j\}=\{1,3\}$, and hence
\[
I=(x_1^a,x_2^a)^2+(x_2^a,x_3^a)^2+(x_1^a,x_3^a)^2+(x_3^{2a},\ldots, x_n^{2a}).
\]
Since we assume that $I$ is not a proper pseudo-Frobenius power, it follows that $a=1$ and $d=2$, contradicting our assumption that $d\geq 3$.
\end{proof}

\begin{Corollary}
\label{cm}
Let $I\subset  K[x_1,\ldots,x_n]$  be a Freiman ideal of height $n$. Then $F(I)$ is Cohen--Macaulay.
\end{Corollary}

\begin{proof} By Theorem~\ref{Thm3} we know that $r(I)=1$. Therefore, since $I$ is equimultiple, the assertion follows from Corollary~\ref{fishsticks}. Here we give an alternative  proof using Theorem~\ref{pseudo}: We first notice that if $I$ is a  graded ideal and $J=(I, f)$ where $f$ is a homogeneous polynomial and a non-zerodivisor modulo $I$. Then $F(J)$ is a polynomial ring over $F(I)$ in one indeterminates. Therefore, $F(I)$ is Cohen--Macaulay if and only if $F(J)$ is Cohen--Macaulay. Thus,  in order to prove that $F(I)$ is Cohen--Macaulay for a Freiman ideal of height $n$ in $K[x_1,\ldots,x_n]$, it is enough, due to Theorem~\ref{pseudo},  to show that $F(L)$ is Cohen--Macaulay for $L=(x_1^a,x_2^a,x_3^a)^2$  or $L=(x_1^a,x_2^a)^r$. In both cases $F(L)$ is just a Veronese ring, which is known to be Cohen--Macaulay, see for example \cite{BH}.
\end{proof}

If in Corollary~\ref{cm} we do not suppose that $\height(I) =n$,  is it still true that $F(I)$ Cohen--Macaulay?

\begin{Corollary}
\label{product}
Let $I,J\subset K[x_1,\ldots,x_n]$ be an equigenerated monomial ideals of height $n$. Then the  following holds:
\begin{enumerate}
\item[(a)] If $n=3$, then $IJ$ is Freiman, if and only if $I=J=(x_1^d, x_2^d, x_3^d)$ for some integer $d\geq 1$.
\item[(b)] If $n\geq 4$, then $IJ$  is not Freiman.
\end{enumerate}
\end{Corollary}

\begin{proof}
(a) If $I=J=(x_1^d, x_2^d, x_3^d)$ for some integer $d\geq 1$, then $IJ$ is Freiman by Theorem~\ref{pseudo}.

Conversely, if $IJ$ is Freiman.  Let $I=(x_1^{d_1},x_2^{d_1}, x_3^{d_1},\ldots)$, $J=(x_1^{d_2},x_2^{d_2},x_3^{d_2},\ldots)$ with  $d_1\geq d_2\geq 1$ and $d_1\neq 1$. Suppose $I$ or $J$ contains a monomial generator $u$  with $\supp(u)\geq 2$. Suppose $u\in G(I)$ with $\supp(u)>2$. Let $v\in G(J)$. Then $uv\in G(IJ)$ and $\supp(uv)>2$. This contradicts (\ref{venlo}), since we assume that  $IJ$ is Freiman.

Now suppose that $\supp(u)= 2$  for  $u\in G(I)$ or $u\in G(J)$.  We assume that $u\in G(I)$   and $u=x_i^ax_j^{d_1-a}$ with $i\neq j$. Since $n=3$, we may choose the integer  $k$ with $1\leq k\leq 3$ and $k\neq i,j$. Then $x_i^ax_j^{d_1-a}x_k^{d_2}\in G(IJ)$, contradicting (\ref{venlo}).

It follows that  $I=(x_1^{d_1},x_2^{d_1}, x_3^{d_1})$ and $J=(x_1^{d_2},x_2^{d_2},x_3^{d_2})$.   We show that $d_1=d_2$. This then  yields the desired conclusion. Indeed, suppose that $d_2\neq d_1$. Then $x_1^{2d_1}x_2^{2d_2}\in (IJ)^2\setminus M(IJ)$ where $M=(x_1^{d_1+d_2},x_2^{d_1+d_2}, x_3^{d_1+d_2})$,  contradicting our assumption that $IJ$ is   Freiman.

(b) If  $IJ$  is   Freiman,  then, by Theorem~\ref{pseudo}, we have $IJ=(x_1^a,x_2^a,x_3^a)^2+(x_4^{2a},\ldots, x_n^{2a})$ or $IJ=(x_1^a,x_2^a)^r+(x_3^{ar},\ldots, x_n^{ar })$ with $a\geq 1,r>1$. If $IJ=(x_1^a,x_2^a,x_3^a)^2+(x_4^{2a},\ldots, x_n^{2a})$, then $x_1^ax_2^{a}x_3^{a}\in I^2J\subset IJ$, contradiction (\ref{venlo}). If
 $IJ=(x_1^a,x_2^a)^r+(x_3^{ar},\ldots, x_n^{ar })$, then  $x_1^ax_2^{ar-a}\in IJ\subset I$.
 Since $x_3^{ar}\in J$, it  follows that $x_1^ax_2^{ar-a}x_3^{ar}\in IJ$, contradiction (\ref{venlo}).
\end{proof}

The reader may wonder why in Corollary~\ref{product} we did not consider the case $n=2$. The reason is that there is no simple answer in this case.
For example, for any integer $d\geq 4$ there are many different ideals $I\subset K[x_1,x_2]$ of degree  $d$ such that $I^2=(x_1,x_2)^{2d}$ (which is Freiman).

But we can give an answer, if we assume that $I$ and $J$ are Freiman. Since any Freiman ideal in two variables is of the form $(x_1^a,x_2^a)^r$ (see Theorem~\ref{pseudo}), the following results provides the required answer.

\begin{Proposition}
\label{zhumustprove}
Let $I=(x_1^a,x_2^a)^r$ and $J=(x_1^b,x_2^b)^s$. Then $IJ=(x_1^c,x_2^c)^t$ for some $c$ and $t$ if and only if  $b=ka$ for some positive integer  $k$ and $r\geq k-1$.
\end{Proposition}
\begin{proof}  We may assume that  $a\leq b$. The  generating set of $IJ$ is $$G(IJ)=\{x_1^{ar+bs-(ai+bj)}x_2^{ai+bj}\mid i=0,\ldots,r, j=0,\ldots,s\}.$$  Thus $IJ=(x_1^c,x_2^c)^t$, if and only if $A=B$, where
\[
A=\{ai+bj \mid i=0,\ldots,r, j=0,\ldots,s\} \text{ and } B=\{cl\mid l=0,\ldots,t\}.
\]
Since the smallest nonzero element of $A$ is $a$ and the smallest nonzero element of $B$ is $c$, it follows that $c=a$.

Since  $b\in A$  and  $A=B$, and all elements of $B$ are multiples of $a$, we obtain that $b=ka$ for some integer $k\geq 1$.

Since $ar+kas=at$, we have $r+ks=t$. Therefore, $t>k$, and hence $(k-1)a\in B$. Suppose that $r<k-1$, then this implies that $(k-1)a\not\in A$, because it cannot be written in the form $ia+kaj$ with  $i=0,\ldots,r$ and $j=0,\ldots,s$. Hence we must have that $r\geq k-1$.

Conversely, if   $b=ka$ for some integer $k\geq 1$ and $r\geq k-1$, then the  generating set of $IJ$ is $$G(IJ)=\{x_1^{a(r+ks)-a(i+kj)}x_2^{a(i+kj)}\mid i=0,\ldots,r, j=0,\ldots,s\}.$$
Let $C=\{i+kj \mid i=0,\ldots,r, j=0,\ldots,s\}\text{ and } D=\{l\mid l=0,\ldots,r+ks\}$. It is clear that $C\subset D$. We have to show that $C=D$, and this is obviously the case if $r\geq k-1$.
\end{proof}

\section{Special classes of Freiman  ideals}
\label{sec3}

\subsection{Principal Borel ideals}

Let $S= K[x_1,\ldots,x_n]$ be the polynomial ring in $n$ indeterminates over a field $K$. A monomial ideal $I\subset S$ is called {\em strongly stable}, if for all $u\in G(I)$ and all $j\in \supp(u)$ it follows that $x_i(u/x_j)\in  I$ for all $i<j$.  Given monomials $u_1,\ldots,u_m\in S$,  there exists a unique smallest strongly stable ideal, denoted $B(u_1,\ldots,u_m)$,  which contains the monomials $u_1,\ldots,u_m$. The monomial ideal $I$ is called a {\em principal Borel ideal},  if
$I=B(u)$ for some monomial $u\in S$.

\medskip
Let $u \in S$ be a monomial. We set $m(u)=\max\{j \;\mid \; j\in \supp(u)\}$.

\begin{Lemma}
\label{analytic spread}
Let $I=B(u_{1},\ldots, u_{m})$,  where  $u_{1},\ldots, u_{m} \in  S$ are   monomials of same degree. Then $\ell(I)=\max\{m(v)\mid v\in G(I)\}$.
\end{Lemma}
\begin{proof} Let $C$ be the integer matrix whose rows correspond to the exponent vectors of the monomials in $G(I)$. Then $\ell(I)=\dim F(I)=\rank(C)$, see
\cite[Lemma 10.3.19]{HH1}.
Let $s=\max\{m(v)\mid v\in G(I)\}$. Then  each $j$th column of $C$ with $j>s$ is zero. Therefore, $\rank(C)\leq s$. Now let $v\in G(I)$ with $m(v)=s$. Since $I$ is strongly stable, the monomials $v_i= x_i(v/x_s)$ belong to $G(I)$ for $=1,\ldots,s$, The exponent vectors of these monomials are linearly independent. Therefore, $\rank(C)\geq s$.
  \end{proof}

\begin{Theorem}\label{Borel}
Let $K$ be a field and $S=K[x_1,\ldots,x_n]$ be the polynomial ring in  $n$ indeterminates over $K$.
\begin{enumerate}
\item[(a)] Let $u=x_ix_j$ with $i\leq j$  be a monomial of degree $2$ in $S$. Then $B(u)$ is Freiman if and only if
$j\leq 3$, or $j>3$ and $i\leq 2$.

\item[(b)] Let  $i, j$ and $d$ be integers such that $u=x_i^{d-1}x_j$ with $d\geq 3$  and $1\leq i\leq j\leq n$. Then  $B(u)$  is Freiman if  $i\leq 2$.
\end{enumerate}
\end{Theorem}

\begin{proof}
(a) Let $I=B(u)$. By Remark~\ref{useful}(b) we may assume that  $u=x_ix_n$,  then
$$I=(x_{1},\ldots,x_{i})(x_{1},\ldots,x_{n})=(x_{1},\ldots,x_{i})^2+(x_{1},\ldots,x_{i})(x_{i+1},\ldots,x_{n})$$
and
\begin{eqnarray*}
 I^2&=&(x_{1},\ldots,x_{i})^{2}(x_{1},\ldots,x_{n})^{2}\\
 &=&(x_{1},\ldots,x_{i})^{4}+(x_{1},\ldots,x_{i})^{3}(x_{i+1},\ldots,x_{n})+(x_{1},\ldots,x_{i})^{2}(x_{i+1},\ldots,x_{n})^{2}.
 \end{eqnarray*}
Since  $G((x_{1},\ldots,x_{i})^2)\sect G((x_{1},\ldots,x_{i})(x_{i+1},\ldots,x_{n}))=\emptyset$, we obtain that
 $$\mu(I)={i+1\choose 2}+i(n-i)=ni-\frac{i^2}{2}+\frac{i}{2}.$$
A similar argument shows that
\begin{eqnarray*}
\mu(I^2)&=&{i+3\choose 4}+{i+2\choose 3}(n-i)+{i+1\choose 2}{n+1-i\choose 2}\\
&=&\frac{i(i+1)}{4!}(3i^3-6i^2-3i+6-8ni^2+6ni+6n^2i+14n+6n^2).
\end{eqnarray*}

By  using  Lemma~\ref{analytic spread}, we then get
 \begin{eqnarray*}
\Delta (I)&=&\mu(I^2)-[\ell(I)\mu(I)-{\ell(I)\choose 2}]\\
&=&\frac{i(i+1)}{4!}(3i^3-6i^2-3i+6-8ni^2+6ni+6n^2i+14n+6n^2)\\
&-&n(ni-\frac{i^2}{2}+\frac{i}{2})+{n\choose 2}\\
&=&\frac{i}{4!}[6n^2(i-3)+n(-8i^2+18i+2)+(3i^3-6i^2-3i+6)]+{n\choose 2}\\
&=&\frac{1}{4!}[6n^2(i^2-3i+2)+n(-8i^3+18i^2+2i-12)+(3i^4-6i^3-3i^2+6i)]\\
&=&\frac{(i-1)(i-2)}{4!}[6n^2-2n(4i+3)+3i(i+1)].
\end{eqnarray*}
 If $i=1$ or $i=2$, then $\Delta (I)=0$, i.e., $I$ is  Freiman.

Now, we assume that $i\geq 3$. Put $f(n,i)=6n^2-2n(4i+3)+3i(i+1)$. Then $$f(n,3)=6(n-2)(n-3).$$
It follows that $I$ is  Freiman, if and only if $n=i=3$.

If $i\geq 4$, then $f(n,i)$ is a strictly monotonic increasing function for $n\geq i$.
Notice that $$f(i,i)=6i^2-2i(4i+3)+3i(i+1)=i(i-3)>0.$$
Therefore, $f(n,i)>0$ for all $n\geq i$, this shows that $\Delta (I)>0$ for all $4\leq i\leq n$.
This completes the proof of (a).

\medskip
(b)  As in (a) we may assume that $j=n$. If $u=x_1^{d-1}x_n$, then $I=x_1^{d-1}(x_1,\ldots,x_n)$. Thus $\mu(I)=\mu(J)$,
$\mu(I^2)=\mu(J^2)$ where $J=(x_1,\ldots,x_n)$. This implies that $I$ is Freiman,  see Theorem \ref{lexsegment}.

If $u=x_2^{d-1}x_n$, then $$I=(x_1,x_2)^{d-1}(x_1,\ldots,x_n)=(x_1,x_2)^{d}+(x_1,x_2)^{d-1}(x_3,\ldots,x_n)$$
and   \begin{eqnarray*}
I^2&=&(x_1,x_2)^{2d-2}(x_1,\ldots,x_n)^2\\
&=&(x_1,x_2)^{2d}+(x_1,x_2)^{2d-1}(x_3,\ldots,x_n)+(x_1,x_2)^{2d-2}(x_3,\ldots,x_n)^2.
\end{eqnarray*}
As in part (a) we see that
\begin{eqnarray*}
\mu(I^2)&=&\mu((x_1,x_2)^{2d})+\mu((x_1,x_2)^{2d-1}(x_3,\ldots,x_n))\\
&+&\mu((x_1,x_2)^{2d-2}(x_3,\ldots,x_n)^2).
\end{eqnarray*}
Hence
\begin{eqnarray*}
\mu(I^2)&=&(2d+1)+2d(n-2)+(2d-1){n-1\choose 2}\\
&=&(2d-1){n\choose 2}+n.
\end{eqnarray*}
It follows that
\begin{eqnarray*}
\Delta(I)&=& \mu(I^2)-[n \mu(I)-{n\choose 2}]\\
&=&(2d-1){n\choose 2}+n-n[(d+1)+d(n-2)]+{n\choose 2}\\
&=&0.
\end{eqnarray*}
\end{proof}

Note that for any monomial $u\in S$ it follows that $B(u)$ is Freiman if and only if $B(x_1^ku)$ is Freiman. Hence, if we want to classify the  principal Borel ideals $B(u)$  which are Freiman, we may assume that $x_1$ does not divide $u$.
We expect that for any monomial $u\in S$ of degree $d>2$ such that $x_1$ does not divide $u$,  the principal Borel ideal  $B(u)$ is Freiman if $u=x_2^{d-1}x_j$ with $j\geq 2$. For the case that $n=3$, this is shown in the next result.

\begin{Theorem} \label{Thm4} Let   $u=x_{1}^{a_{1}}x_{2}^{a_{2}}x_{3}^{a_{3}}\in K[x_{1},x_{2},x_{3}]$ with $a_{i}\geq 0$ for $i=1,2$ and $a_{3}\geq 1$.
Then $B(u)$ is Freiman if and only if $a_{3}=1$,   or $a_{3}=2$ and $a_{2}=0$.
\end{Theorem}

\begin{proof} Let $I=B(u)$,  then $\mu(I)=\mu(J)$, $\mu(I^2)=\mu(J^2)$, where $J=B(x_{2}^{a_{2}}x_{3}^{a_{3}})$. Note that
$$J=(x_{1},x_{2})^{a_{2}}(x_{1},x_{2},x_{3})^{a_{3}}=\sum\limits_{i=0}^{a_{3}}(x_{1},x_{2})^{a_{2}+i}(x_3)^i.$$
This implies that
\[
\mu(I)=\sum\limits_{i=0}^{a_{3}}{a_{2}+1+i\choose a_{2}+i}=\sum\limits_{i=0}^{a_{3}}(a_{2}+1+i)=a_{2}(a_{3}+1)+{a_{3}+2\choose 2}.
\]
Similarly, \[
\mu(I^2)=\sum\limits_{i=0}^{2a_{3}}{2a_{2}+1+i\choose 2a_{2}+i}=\sum\limits_{i=0}^{2a_{3}}(2a_{2}+1+i)=2a_{2}(2a_{3}+1)+{2a_{3}+2\choose 2}.
\]
By Lemma \ref{analytic spread}, it follows that
\begin{eqnarray*}\Delta (I)&=&2a_{2}(2a_{3}+1)+{2a_{3}+2\choose 2}-3[a_{2}(a_{3}+1)+{a_{3}+2\choose 2}]+3\\
&=&a_{2}(a_{3}-1)+{a_{3}-1\choose 2}.
\end{eqnarray*}
Therefore,  $I$ is Freiman if and only if $a_{3}=1$,   or $a_{3}=2$ and $a_{2}=0$.
\end{proof}

\subsection{Hibi ideals}

Fix a field $K$,  and let $P$ be a finite poset. A poset $C$ is called a {\em chain}, if $C$ is totally ordered, that is if any two elements of $C$ are comparable.  By definition, the length of the chain $C$ is equal to $|C|-1$. The {\em rank of $P$}, denoted $\rank(P)$,  is the maximal length of a chain in $P$. We always have $\rank(P)\leq |P|-1$, and equality holds, if and only if $P$ itself is a chain.

Attached to $P$   we consider  a monomial ideal $H_P$ in the polynomial ring  $S=K[\{x_p,y_p\}_{p\in P}]$. This ideal is called the {\em Hibi ideal} of the poset $P$. To define $H_P$,  let $\mathcal{I}(P)$ be the set of poset ideals of $P$. Recall that a subset $I\subset P$ is called a {\em poset ideal} of $P$,  if for all $p,q\in P$ with $p\in P$ and $q\leq p$, it follows that $q\in P$.  Now the Hibi ideal associated to $P$ is the squarefree monomial ideal
\[
H_{P}=(\{u_{I}\}_{I\in \mathcal{I}(P)}),  \quad \text{where}\quad  u_{I}=(\prod\limits_{p\in \mathcal{I}(P)}x_p)(\prod\limits_{p\in P\setminus \mathcal{I}(P)}y_p).
\]
Note that each $u_I$ is a
 squarefree monomial of  degree $|P|$.

\begin{Theorem}
\label{hibifreiman}
Let $P$ be a finite poset. Then $H_P$ is  Freiman   if and only if there exists $p\in P$ such that the subposet $P\setminus \{p\}$ of $P$  is a chain.
\end{Theorem}

\begin{proof}
By Remarks \ref{useful}(a) it suffices to show that $h_2=0$. Since $F(H_P)$ is Cohen--Macaulay (see  \cite{Hi}), it follows  that $h_2=0$ if and only if $h_i=0$ for all $i\geq 2$. Therefore, $h_2=0$ if and only if for the $a$-invariant of $F(H_P)$ we have $a(F(H_P))\leq 1-\ell(H_P)$. It is known that  $\ell(H_P)=|P|+1$  and that $a(F(H_P))=-\rank P-2$ (see \cite{EHM}). It follows that $H_P$ is a Freiman ideal if and only if $\rank(P)\geq |P|-2$.

Now suppose that there exists $p\in P$ such that $P\setminus \{p\}$ is a chain. Then $|P|-2=\rank(P\setminus \{p\})\leq \rank(P)$. Therefore, $H_P$ is Freiman.

Conversely, if $H_P$ is Freiman, then $\rank(P)\geq |P|-2$. Suppose $P\setminus \{p\}$ is not a chain for any $p\in P$. Then $P$ is not a chain. Let $C\subset P$ be a chain with  $\rank(C)=\rank(P)$. Since $P$ is not a chain, there exists $p\in P\setminus C$. Then
\[
|P|-2=|P\setminus \{p\}|-1\geq \rank(P\setminus\{p\})=\rank(P)\geq |P|-2.
\]
Therefore, $|P|=\rank(P)+2$. This means that $P=C\union \{q\}$ with $q\not \in C$. Then $P\setminus\{q\}$ is a chain, a contradiction.
\end{proof}

\subsection{Ideals of Veronese type}

Given positive integers  $n$, $d$ and a sequence $\ab$ of integers $1\leq a_1\leq a_2\leq \cdots \leq a_n\leq d$ with $\sum_{i=1}^na_i>d$, one defines the monomial ideal  $I_{\ab,d}\subset S=K[x_1,\ldots,x_n]$ with
\[
G(I_{\ab,d})=\{x_1^{b_1}x_2^{b_2}\cdots x_n^{b_n}\; \mid \; \sum_{i=1}^nb_i=d \text{ and  $b_i\leq a_i$ for $i=1,\ldots,n$}\}.
\]
By \cite[Corollary 2.2]{NH},   $F(I_{\ab,d})$ is a Cohen--Macaulay of dimension $n$.

\medskip
If $a_i=d$  for $i=1,\ldots,n$, then $I_{\ab,d}=(x_1,\ldots,x_n)^d$, and we have seen in Theorem~\ref{lexsegment} that $(x_1,\ldots,x_n)^d$ is Freiman if and only if  $n\leq 2$, $d=1$  or $n=3$ and  $d=2$.

\medskip
If $a_i=1$ for $i=1,\ldots,n$, then $I_{\ab,d}$ is the so-called squarefree Veronese ideal which we also denote $I_{n,d}$. In this case we have

\begin{Theorem}\label{Veronese2}
  $I_{n,d}$ is  Freiman  if and only if  $d=1$,  $d=n-1$.
\end{Theorem}
\begin{proof} Since $F(I_{n,d})\iso F(I_{n,n-d})$ for all $1\leq d\leq n-1$, it follows that $I_{n,d}$ is Freiman if and only if $I_{n,n-d}$ is Freiman.
If $d=1$, then $I_{n,d}=(x_1,\ldots,x_d)$ is Freiman by Theorem \ref{lexsegment}. Therefore,  $I_{n,d}$
is also Freiman for $d=n-1$.

\medskip
Conversely,  assume  that  $I_{n,d}$  is Freiman.  It is known from \cite[Corollary 2.2]{NH} that $F(I_{n,d})$ is Cohen--Macaulay. Therefore, Corollary~\ref{level} implies that $I_{n,d}$ is level. By \cite[Theorem  2.1]{HHV}, this implies that  $d=1$, $d=n-1$,  or $d\geq 2$ and $n=2d-1$, $n=2d$ or $n=2d+1$.
 Since $F(I_{n,d})\iso F(I_{n,n-d})$ for all $1\leq d\leq n-1$,  we may assume that  $n\geq 2d$,  and it is enough to show that if $d\geq 2$ and $n=2d$ or $n=2d+1$, then $I_{n,d}$ is not Freiman.

To see this, first observe that
\[
\mu(I_{n,d})={n\choose d}\ \text{and}\ \mu(I_{n,d}^2)=\sum\limits_{i=0}^{d}{n\choose 2i}{n-2i\choose d-i}.
\]
The first equation is obvious. For the second  equation, we notice that  ${n\choose 2i}{n-2i\choose d-i}$ counts the number of monomials
$(x_{i_1}\cdots x_{i_d})(x_{j_1}\cdots x_{j_d})\in I_{n,d}^2$ with $$|\{i_1,\ldots,i_d\}\cap \{j_1,\ldots,j_d\}|=i.$$

By \cite[Corollary 2.2]{NH},   $\ell(I_{n,d})=n$ for any $1\leq d\leq n-1$, it follows that
 $$\Delta(I_{n,d})=\mu(I_{n,d}^2)-n\mu(I_{n,d})+{n\choose 2}.$$

If $d\geq 2$ and $n=2d$, then
\begin{eqnarray*}
\Delta(I_{2d,d})&=&\mu(I_{2d,d}^2)-2d\mu(I_{2d,d})+{2d\choose 2}\\
&=&{2d\choose d}+\sum\limits_{i=1}^{d}{2d\choose 2i}{2d-2i\choose d-i}-2d{2d\choose d}+{2d\choose 2}\\
&=&{2d\choose d}+{2d\choose 2}{2d-2\choose d-1}-2d{2d\choose d}+\sum\limits_{i=2}^{d}{2d\choose 2i}{2d-2i\choose d-i}+{2d\choose 2}\\
&=&{2d-1\choose d}(d^2-4d+2)+\sum\limits_{i=2}^{d}{2d\choose 2i}{2d-2i\choose d-i}+{2d\choose 2}\\
&>&0.
\end{eqnarray*}
It follows that $I_{2d,d}$ is not Freiman.

If $d\geq 2$ and $n=2d+1$, then
\begin{eqnarray*}
\Delta(I_{2d+1,d})&=&\mu(I_{2d+1,d}^2)-(2d+1)\mu(I_{2d+1,d})+{2d+1\choose 2}\\
&=&{2d+1\choose d}+\sum\limits_{i=1}^{d}{2d+1\choose 2i}{2d+1-2i\choose d-i}-(2d+1){2d+1\choose d}+{2d+1\choose 2}\\
&=&{2d+1\choose d}+{2d+1\choose 2}{2d-1\choose d-1}-(2d+1){2d+1\choose d}\\
&+&\sum\limits_{i=2}^{d}{2d+1\choose 2i}{2d+1-2i\choose d-i}+{2d+1\choose 2}\\
\end{eqnarray*}
\begin{eqnarray*}
&=&\frac{(2d+1)(d-3)}{2}{2d\choose d-1}+\sum\limits_{i=2}^{d}{2d+1\choose 2i}{2d+1-2i\choose d-i}+{2d+1\choose 2}\\
&>&0.
\end{eqnarray*}
This implies that $I_{2d+1,d}$ is not Freiman.
\end{proof}

We conclude this paper with the classification of all Freiman ideals for  another family of ideals of Veronese type.

\begin{Theorem}\label{Veronese3}
 Let $d, n\geq 2$ be positive integers and  $a_i=d-1$ for $i=1,\ldots,n$, then $I_{\ab,d}$ is  Freiman  if and only if $n=2$, or $n=3$ and  $d=2$.
\end{Theorem}
\begin{proof} In  \cite[Corollary 2.11]{Ka}, Katzman showed  that the multiplicity   of $F(I_{\ab,d})$ is given by the fomula
\[
e(F(I_{\ab,d}))=\sum\limits_{S\in \Mc}(-1)^{|S|}(d-\sum_{i\in S}a_i)^{n-1},
\]
where $\Mc$ is the set of subsets $S$ of $\{1,\dots,n\}$  with $\sum_{i\in S}a_{i}<d$.

Since we assume that  $a_i=d-1$ for $i=1,\ldots,n$, it follows from the above formula by Katzman that
 \[
 e(F(I_{\ab,d}))=d^{n-1}-n.
 \]

 Since $F(I_{\ab,d})$ is Cohen--Macaulay, Proposition~\ref{hyper} implies $I_{\ab,d}$ is  Freiman if and only if  $F(I_{\ab,d})$ has minimal multiplicity.  This is the case   if and only if
\[
\mu(I_{\ab,d})= e(F(I_{\ab,d}))+\dim(F(I_{\ab,d}))-1.
\]
Since  $\dim(F(I_{\ab,d}))=n$ (see  \cite[Corollary 2.2]{NH}), this is equivalent to the equation
\[
{n+d-1\choose d}-n=d^{n-1}-1.
\]
Let $f(n,d)=d^{n-1}-{n+d-1\choose d}+n-1$. Then   $f(n,d)=0$,   if and only if  $I_{\ab,d}$ is Freiman.

Note that $$f(2,d)=d-{d+1\choose d}+2-1=0,$$
and
$$f(3,d)=d^2-{d+2\choose d}+3-1=\frac{(d-1)(d-2)}{2}.$$
Now, we prove that if $n\geq 4$, then $f(n,d)>0$ for any $d\geq 2$.
We  prove this by induction on $n$.
$$f(4,d)=d^3-{d+3\choose d}+4-1=\frac{(d-1)(5d^2-d-12)}{6}>0,$$

and $f(n,d)=d^{n-1}-{n+d-1\choose d}+n-1>0$, by induction hypothesis. Then
  \begin{eqnarray*}
  f(n+1,d)&=&d^{n}-{n+d\choose d}+n\\
  &>&d[{n+d-1\choose d}-n+1]-{n+d\choose d}+n\\
  &=&(d-1){n+d-1\choose d}-{n+d-1\choose d-1}-n(d-1)+d\\
  &=&\frac{\prod\limits_{i=1}^{d-1}(n+i)-d!}{d!}(n(d-1)-d).
   \end{eqnarray*}
  Since $\sum_{i=1}^na_i>d$ and $a_i=d-1$ for $i=1,\ldots,n$, it follows that $n(d-1)-d>0$. Therefore,
   $f(n,d)>0$ for any $d\geq 2$, and hence  $f(n,d)=0$ if and only if $n=2$, or $n=3$ and  $d=2$.
\end{proof}

\medskip
\noindent
{\bf Acknowledgement.}  This paper is supported by the National Natural Science Foundation of
China (11271275) and by the Foundation of the Priority Academic Program Development of Jiangsu Higher Education Institutions.

\end{document}